\documentclass[12pt, reqno]{amsart}
\usepackage{amsthm}
\usepackage{amssymb}
\usepackage{amsmath}

\def\1{\hbox{1\kern-.35em\hbox{1}}}

\newtheorem{theorem}{Theorem}[section]
\newtheorem*{theorem*}{Theorem}
\newtheorem{lemma}[theorem]{Lemma}

\newtheorem*{proposition*}{Proposition}
\newtheorem{corollary}[theorem]{Corollary}

\theoremstyle{definition}
\newtheorem{definition}[theorem]{Definition}
\newtheorem{example}[theorem]{Example}

\theoremstyle{remark}
\newtheorem{remark}[theorem]{Remark}

\numberwithin{equation}{section}

\newcommand{\bea}{\begin{eqnarray}}
\newcommand{\eea}{\end{eqnarray}}
\newcommand{\be}{\begin{eqnarray*}}
\newcommand{\ee}{\end{eqnarray*}}

\newcommand{\Z}{{\mathbb Z}}

\newcommand{\C}{{\mathbb C}}

\newcommand{\fb}{{\mathfrak b}}
\newcommand{\fg}{{\mathfrak g}}
\newcommand{\fh}{{\mathfrak h}}

\newcommand{\fn}{{\mathfrak n}}
\newcommand{\fp}{{\mathfrak p}}
\newcommand{\fq}{{\mathfrak q}}

\newcommand{\fu}{{\mathfrak u}}

\newcommand{\gl}{{\mathfrak{gl}}}
\newcommand{\osp}{{\mathfrak{osp}}}

\newcommand{\id}{{\rm id}}
\newcommand{\Hom}{{\rm Hom}}

\newcommand{\cO}{{\mathcal O}}

\def\TC#1{{\mathcal C}^{\rm Trun}_{#1}}
\def\NC#1{{\mathcal C}^{\rm Norm}_{#1}}

\newcommand{\cL}{{\mathcal L}}

\def\soe{\preccurlyeq}

\def\LE{_{\rm L}}
\def\RI{_{\rm R}}
\def\BO{_{\rm B}}

\def\a{\alpha}

\def\ch{{\rm ch{\ssc\,}}}
\def\th{\theta}
\def\b{\beta}
\def\d{\delta}
\def\D{\Delta}
\def\g{\gamma}
\def\G{\Gamma}
\def\l{\lambda}
\def\L{\lambda}
\def\si{\sigma}

\def\sc{\scriptstyle}
\def\ssc{\scriptscriptstyle}
\def\dis{\displaystyle}

\def\ul{\underline}

\def\D{\Delta}

\def\Lra{\Longleftrightarrow}

\def\rb{\raisebox}

\def\dd#1#2#3#4{{\widehat c}{\ssc\,}_{#4}(#2)}

\def\PP{P}
\def\Dr{P_{\rm Lex}^\l}

\def\CO{{\cal O}}

\def\N{\mathbb{N}}

\def\cal#1{{\mathfrak{#1}}}

\def\es{\varepsilon}

\def\Hom{{\rm Hom}}

\def\cp{}
\def\tt{{\rm aty}}
\def\ot{{\rm typ}}
\def\nn{{n}}
\def\mm{{m}}

\def\ITEM{\item[\mbox{$\bullet$}]}

\def\bi{{\ul{i}}}
\def\bj{{\ul{j}}}

\def\Sr{{\rm Sym}}

\def\ZZ{\Z_{{\ssc++}}}

\def\a{\alpha}\def\b{\beta}
\def\equa#1#2{
\begin{equation}\label{#1}#2\end{equation}}
\def\equan#1#2{$$#2$$}
\numberwithin{equation}{section}
\begin{document}

%

\title[Character and dimension formulae for ${\mathfrak{q}}_n$]
{Character and dimension formulae for\\  queer Lie superalgebra}
\author[Yucai Su]{Yucai Su}
\address{Department of Mathematics, Tongji University, Shanghai 200092, China}
\email{ycsu@tongji.edu.cn}
\author[R. B. ZHANG]{R. B. ZHANG}
\address{School of Mathematics and Statistics,
University of Sydney, NSW 2006, Australia}
\email{rzhang@maths.usyd.edu.au}

\begin{abstract}
Closed formulae  are constructed for the characters and dimensions of the
finite dimensional simple modules of the queer Lie superalgebra $\fq(n)$.
This is achieved by refining Brundan's algorithm for computing simple $\fq(n)$-characters.
\end{abstract}
\subjclass[2010]{Primary 17B10.}
\thanks{This work was supported by the Australian Research
Council (grant no.~DP0986551), NSF of China (grant
no.~10825101), SMSTC (grant no.~12XD1405000) and
the Fundamental Research Funds for the Central Universities of China.}
\maketitle

%
%

\section{Introduction}

There has been considerable progress in the representation theory \cite{Kac2} of
Lie superalgebras \cite{Kac, Sch} in recent years.
In \cite{B}, Brundan reformulated Serganova's Kazhdan-Lusztig approach \cite{Se96}  to
the parabolic category $\cO$ of the general linear superalgebra
$\gl(m|n)$ using quantum group techniques,  obtaining a very practicable algorithm
for computing the generalized Kazhdan-Lusztig polynomials.
This enabled him to prove  the conjecture of \cite{VZ} on the composition
factors of Kac modules. By using this algorithm, the present authors \cite{SZ2007}
derived closed character and dimension formulae for the finite dimensional
simple $\gl(m|n)$-modules.  An algorithm was developed by Serganova and Gruson
\cite{Se98, GS}, which enables one to compute the characters of the finite dimensional simple modules for
the orthosymplectic Lie superalgebras $\osp(m|2n)$. In the special case of $\osp(m|2)$,
the characters and dimensions of the simple modules were worked out explicitly \cite{SZ2012-2}.

Properties of the parabolic categories $\cO$ for simple Lie superalgebras
(relative to the maximal parabolic subalgebra with a purely even Levi subalgebra)
have also been investigated extensively.
It was discovered in \cite{CZ1, CZ2} that the representation theory
of the classical Lie superalgebras $\gl(m|n)$ and $\osp(m|2n)$
was related to that of infinite dimensional classical Lie algebras in some precise way.
This led to the super duality conjecture in \cite{CWZ}, which was proven in \cite{CL, BS} for $\gl(m|\infty)$
and generalised to $\osp(m|\infty)$ in \cite{CLW}.  A Jantzen type filtration for parabolic Verma modules
of the type I and exceptional Lie superalgebras was developed \cite{SZ2012-1, SZ2013}.
It was shown that the layers of the Jantzen type filtration
are semi-simple and the multiplicities of their composition factors are determined by
the coefficients of the inverse of Serganova's generalised Kazhdan-Lusztig polynomials.

Penkov and Serganova \cite{PS97} developed an algorithm for
computing the characters of finite dimensional simple modules
for the queer Lie superalgebra $\fq(n)$.
They used a super analogue of the Bott-Borel-Weil theory \cite{PS, PS1} to realise
$\fq(n)$-representations on cohomology groups of vector bundles on super flag  varieties
of the corresponding supergroup $Q_n$. Their algorithm
enables one to understand composition factors of the cohomology groups combinatorially.
Brundan provided an important new input into the problem \cite{B1} by
translating it to a problem about some canonical bases of the quantum
group associated with the infinite dimensional Lie algebra of type $\fb_\infty$.
By working out the combinatorics of the canonical bases,
he obtained a more efficient algorithm for
determining the composition factors.

The algorithms are sufficient in some situations, especially when the problems concerned
require only knowledge of characters of specific simple $\fq(n)$-modules.
However,   for some problems, e.g., if one wants to
determine the dimensions of finite simple modules in a uniform way,
the algorithms are not adequate and a closed formula for the characters will be needed.
At any rate, a character formula is more desirable.

In this paper, we develop a closed formula of Weyl-Kac type for the characters of the finite
dimensional simple modules for the queer Lie superalgebra $\fq(n)$,
and derive from it a formula for the dimensions of the simple modules.
The character formula is obtained by using a refined version of the algorithm of \cite{B1}.
The main results of the paper are presented in Section 4. In particular,
Theorem \ref{main-theo} gives the character and dimension formulae,
and Corollary \ref{coooo} rewrites the formulae in more convenient terms.
This work is a continuation of \cite{SZ2007} but for $\fq(n)$.

The organization of the paper is as follows. In Section 2 we
present some background material on $\fq(n)$,
which will be used throughout the paper.  In Section 3 we develop some combinatorics
for the Euler modules. \cite[Main Theorem]{B1} is reformulated in terms of
weight diagrams, providing an algorithm which is more effective for what we are trying to
do here.   Section 4 contains the main results.

\section{Preliminaries}\label{Preliminaries}
\subsection{The queer Lie superalgebra $\fq(n)$}

A {\it super vector space} over $\C$ is a $\Z_2$-graded vector space $V = V_{\bar0}\oplus V_{\bar1}$, where $\Z_2
 =\{\bar0,\bar1\}$. Denoted by $[v]$ the {\it parity} of $v$ which is either $\bar0$ or $\bar1$ according to whether
 $v\in V_{\bar0}$ or  $v\in V_{\bar1}$.
 A vector $v$ is {\it even} if $[v]=\bar0$ and {\it odd} otherwise.
For two super vector spaces $V$ and $W$, a linear map $f : V\to W$ is homogeneous
of degree $d\in\Z_2$ if $f(V_a )\subset W_{a+d}$ for $a\in\Z_2$. Denote $\Hom(V,W)_d =\{f : V\to W\,|\,f$
is homogeneous of degree $d\}$. Then
  ${\Hom}(V,W)=\Hom(V,W)_{\bar0}\oplus\Hom(V,W)_{\bar1}$ is a super
vector space.
The {\it dual super vector space} of $V$ is defined to be $V^*=\Hom(V,\C)$.
For homogeneous linear maps $f_1 : V_1\to W_1$ and $f_2: V_2\to W_2$,
the {\it tensor product} $f_1\otimes f_2 : V_1\otimes V_2\to W_1\otimes W_2$ is a homogeneous linear map
defined by $(f_1\otimes f_2)(v_1\otimes  v_2) =(-1)^{[f_2][v_1]} f_1(v_1)\otimes
 f_2(v_2)$.
The super vector spaces and the homogeneous mappings above define a category, which
is not abelian. However, if we restrict the Hom-sets to the even mappings, we obtain
an abelian category, denoted by $\CO$. Denote by $\Pi: \CO\to\CO$ the parity change functor, defined as follows: if
$V\in\CO$ we set $(\Pi(V ))_d = V_{d+1}$ for $d\in\Z_2$, and if $f : V\to W$ in $\CO$ we set
$\Pi(f) = f$ as mappings. We denote by
 $\Pi(\C)$ the odd super vector space of dimension one.

For any positive integer $n$, let $\fg = \fq(n)$ denote the {\it queer Lie superalgebra} consisting of the block
matrices of the form
$({}\,^{A\ B}_{B\ A})$
for $A,B\in gl_n(\C)$. The even (resp., odd) subspace $\fg_{\bar0}$ (resp., $\fg_{\bar1}$) of $\fg$
consists of block matrices with
$B = 0$ (resp., $A = 0$). Let
\begin{eqnarray}\label{eq:basis}\label{basis-}
e^{\bar0}_{ij}=\begin{pmatrix}
E_{i j} &0\\ 0 &  E_{i j}
\end{pmatrix},
\quad
e^{\bar1}_{ij}=\begin{pmatrix}
0& E_{i j}\\  E_{i j} & 0\end{pmatrix}
\mbox{ \ for \ \ }1\le i, j\le n,
\end{eqnarray}
denote respectively the basis elements of $\fg_{\bar0}$ and $\fg_{\bar1}$,
where $E_{ij}\in gl_n(\C)$ is the matrix unit with $1$ at position $(i, j)$ and $0$ elsewhere.

Then the supercommutator is defined as follows:
\equa{supercom1}{[e^s_{ij}, e^t_{kl}] =\d_{jk}e^{s+t}_{il}-(-1)^{st}\d_{il}e^{s+t}_{kj},}
where $s,t\in\Z_2$ and $1\le i, j\le n$. Take
$\fh={\rm span}_\C\{e^s_{ii},\,|\,s\in\Z_2,\,1\le i\le n\}$ to be the {\it Cartan Lie subsuperalgebra}.
Then ${\rm sdim}(\fh) = n + n\bar\es$ and $\fh_{\bar0}$ is a Cartan subalgebra of $gl_n(\C)$. The
elements $H_i=e^{\bar0}_{ii}$ (resp., $\hat H_i=e^{\bar1}_{ii}$)
for $1\le i\le n$ form a basis in $\fh_{\bar0}$ (resp., $\fh_{\bar1}$).
Let $\es_i,\, 1\le i\le n$, be the basis of $\fh_{\bar0}^*$ dual to the basis $\{H_i\mid 1\le i\le n\}$ of $\fh_{\bar0}$,
that is, $e_i(H_j)=\delta_{i j}$ for all $i, j$. Then $\fh_{\bar0}^*$ is equipped with
the bilinear form $(\cdot,\cdot)$ defined by
\equa{form-}{(\es_i,\es_j)=\d_{ij}, \quad \forall i, j.}
We have a root space decomposition of $\fg$ with respect to $\fh_{\bar0}$, given by
\equa{root-de}{\fg =\fh\bigoplus\left(\bigoplus_{\a\in\D_{\bar0}}\fg^\a\right)
\bigoplus\left(\bigoplus_{\a\in\D_{\bar1}}\fg^\a\right),}
where $\D_{\bar0}= \{\es_i-\es_j \mid 1\le i, j\le n, \ i\ne j\}$ is the set of even roots,  and
$\D_{\bar1}$ is the set of odd roots,  which is isomorphic to $\D_{\bar 0}$ as set.
Let $\D= \D_{\bar0}\sqcup\D_{\bar1}$ (disjoint union).
We have
$\fg^\a=\C e^s_{ij}$ if $\a=\es_i-\es_j\in\D_s,\,s\in\Z_2$.
Then $\D_{\bar0}$ is the root system of $\fg_{\bar0}= gl_n(\C)$.
 The {\it positive root system} is $\D^+ =\D^+_{\bar0}\sqcup\D^+_{\bar1}$ with
$\D^+_{\bar0}=\{\es_i-\es_j\, |\,1\le i < j\le n\}$ and similarly for $\D^+_{\bar1}$.
The {\it simple root system} is $\Pi^+ =\Pi^+_{\bar0}\sqcup\Pi^+_{\bar1}$ with $\Pi^+_{\bar0}
=\{\es_i-\es_{i+1}\,|\,1\le i\le n-1\}$ and similarly for $\Pi^+_{\bar1}$. Set \equa{rho-0}{\rho_{\bar0}=
\frac12\sum_{\a\in\D^+_{\bar0}}\a=\frac12(n-1,n-3,...,1-n),}
 and $\rho_{\bar1}= \frac12\sum_{\a\in\D^+_{\bar1}}\a=\rho_{\bar0}$
and $\rho=\rho_{\bar0}-\rho_{\bar1}=0$. Note that $\rho_{\bar0}$ corresponds to half of the sum
of the positive roots of the Lie algebra $\fg_{\bar0}\cong gl_n(\C)$.

The {\it Borel subsuperalgebra} of $\fg$ is $\fb=\fh\oplus\fn^+$, where $\fn^+ =\oplus_{\a\in\D^+}\fg^\a$.
Take any parabolic subalgebra $\fp_{\bar0}\supseteq\fb_{\bar0}$ of $\fg_{\bar0}\cong gl_n(\C)$,
and let $\D(\fp_{\bar0})$ be the set of roots of $\fp_{\bar0}$, that is, $\alpha\in \D(\fp_{\bar0})$ if and only if $\fg^\alpha\subset\fp_{\bar0}$.
Let $\D(\fp_{\bar1})$ be a subset of $\D_{\bar1}$ which is isomorphic to $\D(\fp_{\bar0})$, and set $\D(\fp)=\D(\fp_{\bar0})\sqcup\D(\fp_{\bar1})$ (disjoint union).
Then $\fp=\fh\oplus\left(\bigoplus_{\alpha\in\D(\fp)}\fg^\alpha\right)$ is a parabolic subalgebra
containing the Borel subalgebra $\fb$. We may also define the lower triangular Borel subalgebra $\bar{\fb}$
and parabolicc subalgebras $\bar{\fp}\supset\bar{\fb}$.

\subsection{Integral dominant weights}

We will express a weight $\l\in\fh^*_{\bar0}$ in term of its coordinate relative to
the basis $(\es_1, . . . , \es_n)$ as $\l =\sum\limits_{i=1}^n\l_i\es_i$ and simply write
\equa{weight1}{\l =(\l_1, . . . , \l_n)\in\C^n.}
Introduce the following integers
\begin{eqnarray}\label{ZH(l)}
&&z(\l)=\#\{i\,|\,\l_i=0\}\quad \mbox{ (the number of zero entries of $\l$)},\nonumber\\
&&\bar z(\l)=0\mbox{ if $z(\l)$ is even or $1$ else  \quad (the parity of $z(\l)$)},\nonumber\\
&&h(\l)=n-z(\l)\quad \mbox{ (the number of nonzero entries of $\l$)}.
\end{eqnarray}
A weight $\l$ is called {\it integral} if $\l\in\Z^n$, and
{\it integral dominant} if $\l\in\ZZ^n$, where
\equa{Domi-}{\ZZ^n=\{\l\in\Z^n\mid \l_1\ge\cdots\ge\l_n \
\text{and $\l_i=\l_{i+1}$ implies $\l_i=0$ for any $i$}\}.}
Let $W\cong \Sr_n$ be the {\it Weyl group} of $\D_{\bar0}$,
which acts on both $\fh_{\bar0}$ and $\fh^*_{\bar0}$ in the usual way.
As $\rho=0$, the {\it dot action} of $W$ coincides with the usual action.
Define a total order on $\ZZ^n$ lexicographically, namely
\equa{partial-order}{\l<\mu\ \Longleftrightarrow\
\mbox{ for the first $p$ with $\mu_p\ne\l_p$, we have $\mu_p<\l_p$}.}

\def\L{\l}\noindent
\begin{definition}\label{regular-def}
An integral weight $\L$ is called
\begin{enumerate}
\item[(1)]
 {\it regular}
or {\it non-vanishing} (in sense of \cite{HKV, VHKT}) if it is $W$-conjugate to
an integral dominant weight
(which is denoted by $\L^+$ throughout the paper);
\item[(2)]
{\it vanishing} otherwise (since the right-hand side of \eqref{g-Schur-P} vanishes in this case).
\end{enumerate}
\end{definition}
Obviously,
$\L$ is regular if and only if
\equa{l-ru}{\mbox{any nonzero number appears at most once as an entry of $\l$.}}
Let $\L$  in \eqref{weight1} be a regular weight.
A positive root $\g=\es_i-\es_j$ is an {\it atypical root of $\L$} if
$\l_i=-\l_j$.
In case $\lfloor \frac{z(\l)}{2}\rfloor>0$ (where $\lfloor a\rfloor$ denotes the integer part of $a$),
we always take $\lfloor \frac{z(\l)}{2}\rfloor$ atypical roots $\g_p=\es_{m_p}-\es_{n_p},\,
p=1,...,\lfloor \frac{z(\l)}{2}\rfloor$, to satisfy \equa{zero-aty}
{\l_{m_p}=\l_{n_p}=0,\ \ m_{\lfloor \frac{z(\l)}{2}\rfloor}
<\cdots<m_2<m_1<n_1<n_1<\cdots<n_{\lfloor \frac{z(\l)}{2}\rfloor}.}
Furthermore, if $\bar z(\l)=1$, we choose $m_1,n_1$ to be the indices such that
the extra zero entry lies in between, i.e., there exists a unique $k_0$ with $m_1<k_0<n_1$ and $\l_{k_0}=0$.
Denote by $\G_\l$ the set of atypical roots of $\l$
(cf.~(\ref{(lambda-1)}) and (\ref{aty-llll})):
\equa{G-L}
{\G_\l=\{\es_i-\es_j
\,|\,\,\mbox{either $\l_i=-\l_j\ne0$, or else $i=m_p,j=n_p$ for some $p$}\}.
}
Set $r=\#\G_\L.$
We also denote $\#\L=r$, called the {\it degree of
atypicality of $\L$}.
 A weight $\L$
is called
\begin{enumerate}
\ITEM
{\it typical} if $r=0$;
\ITEM
{\it atypical} if $r>0$ (in this case $\L$ is also called
an {\em $r$-fold atypical weight}).
\end{enumerate}
When $\l$ is integral dominant,
we always arrange the atypical roots of $\l$ as
$\g_p=\es_{m_p}-\es_{n_p}$ for $1\le p\le r$ with $m_p<m_{p+1}<n_{p+1}<n_p$ for $1\le p< r$.
Thus,\begin{equation}\label{aty-llll}
\G_\l=\{\g_1<\g_2<\cdots<\g_r\},
\end{equation}
by the order \eqref{partial-order}, and $\l$ has the following form:
\equa{(lambda-1)}
{\mbox{$
\L\!=\!(\L_1,...,
\put(4,11){$\line(0,1){8}$}
\put(4,19){$\line(1,0){150}\,\rb{-4pt}{$\g_r$}\,\line(1,0){30}$}
\L_{\mm_r}
, ...,
\L_i, ...,
\put(4,10){$\line(0,1){5}$}
\put(4,15){$\line(1,0){20}\rb{-5pt}{$\,\g_1\,$}\line(1,0){30}$}
\L_{\mm_1}
, ..., \L_k, ...,
\put(4,10){$\line(0,1){5}$}
\put(4,15){$\line(-1,0){25}$}
\L_{\nn_1}
, ...,
\L_{j}, ...,
\put(4,11){$\line(0,1){8}$}
\put(4,19){$\line(-1,0){40}$}
\L_{\nn_r}
,...,\L_{ n})$},
}

\subsection{Formal characters}

Let $V=\oplus_{\l\in\fh_{\bar0}^*}V_\l$ be a {\em weight module} over $\fg$, where
\equan{weight-space}
{V_\l=\{v\in V\,|\,hv=\l(h)v,\,\forall\,h\in\fh_{\bar0}\}\mbox{ \ with \ }\dim V_\l<\infty,
}
is the weight space of weight $\l$. We denote by $\ch V=\sum_{\l\in\fh_{\bar0}^*}({\rm dim\,}V_\l) e^\l$
its {\em character}.

Given $\l\in\ZZ^n$, there is a unique (up to isomorphism) finite dimensional
simple $\fg$-module $L(\l)$ with highest weight $\l$ such that $\dim({\rm End}_\fg L(\l))=1$ or $2$.
We say that $L(\l)$ is of {\it type M} in the former case, and {\it type Q} in the later case.
For example, the
natural representation $V=L(\es_1)$ of $\fg$ is of type Q, as there exists an odd automorphism
$\theta:v_i\mapsto v_{i+n}$ and $v_{i+n}\mapsto -v_i$ for $1\le i\le n$, where $\{v_i\,|\,1\le i\le 2n\}$
is the natural basis of $V$.
In general, $L(\l)$ is of type M if $h(\l)$ is even, and of type Q otherwise.

For each $\l\in\Z^n$, there exists a unique simple $\fh$-module $\fu(\l)$. Each $e_{i i}^{(0)}$ acts on $\fu(\l)$
by $\l_i \id_{\fu(\l)}$, thus it follows from $[e_{i i}^{(1)}, e_{j j}^{(1)}]=2\d_{i j} e_{i i}^{(0)}$
for all $i, j$ that the actions of all $e_{i i}^{(1)}$ on $\fu(\l)$ generate an irreducible representation
of the Clifford algebra of degree $h(\l)$.  Hence $\dim\fu(\l)= 2^{\lfloor \frac{h(\l)+1}{2}\rfloor}$
and the character of $\fu(\l)$ is given by $2^{\lfloor \frac{h(\l)+1}{2}\rfloor}e^\l$.

Denote by $G=Q_n$ the supergroup with Lie superalgebra $\fq(n)$.
Let $\bar{B}$ be the lower triangular Borel subgroup of $Q_n$
with Lie superalgebra $\bar{\fb}$, and let $\bar{P}(\l)\supseteq\bar{B}$ be
the largest parabolic subgroup such that $\fu(\l)$ can be lifted to a $\bar{P}(\l)$-module.
The sheaf cohomology groups of  the associated
super vector bundle $\cL(\fu(\l))=Q_n\times_{\bar{P}(\l)}\fu(\l)$,
\begin{eqnarray}\label{higher-mod}
H^i(\l) = H^i(G/\bar{P}(\l), \cL(\fu(\l))), \quad i=0, 1, ...,
\end{eqnarray}
admit natural $G$-actions. It is known that $H^i(\l)$ are finite dimensional for all $i$ and vanish for $i\gg 0$.
In particular, $H^0(\l)$ contains a unique simple submodule $L(\l)$,  and $\{L(\l)\}_{\l\in\ZZ^n}$ is the complete set
of pairwise non-isomorphic simple $G$-modules. Here  $L(\l)$ is the highest weight $\fg$-module
with highest weight $\l$ relative to $\fb$ (not $\bar\fb$).

For any $\l\in\ZZ^n$, we define the virtual module $E(\l)$, the {\it Euler module}, by
\begin{equation}\label{Euler-mod}
E(\l)=\sum_{i\ge0}(-1)^i H^i(\l),
\end{equation}
which should be interpreted as an element in the Grothendieck group of the category
of finite dimensional $\fg$-modules. The {\it Euler character} ${\rm ch\,}E(\l)= \sum_{i\ge0}(-1)^i {\rm ch\,}H^i(\l)$ is
given by (cf.~\cite{PS, B1, BK}) the formula
$
{\rm ch\,}E(\l)
=2^{\lfloor\frac{h(\l)+1}{2}\rfloor}P_\l
$
with
\begin{eqnarray}\label{Schur-P}
P_\l &:=&\sum_{w\in \Sr_n/S_\l}
w\left(e^\l\prod_{\stackrel{\ssc 1\le i<j\le n}{\ssc \l_i>\l_j}}
\frac{1+e^{\es_j-\es_i}}{1-e^{\es_j-\es_i}}\right)\nonumber\\
&=&\frac1{\#S_\l}\sum_{w\in \Sr_n}w
\left(e^\l\prod_{\stackrel{\ssc 1\le i<j\le n}{\ssc \l_i>\l_j}}
\frac{1+e^{\es_j-\es_i}}{1-e^{\es_j-\es_i}}\right)
\end{eqnarray}
being {\it Schur's P-function},
where $S_\l$ is the stabilizer of $\l$ in $\Sr_n$, $\#S_\l$ is the size of $S_\l$,
and $\Sr_n/S_\l$ denotes the set of minimal length coset representatives.

\section{Combinatorics of Euler modules}

\subsection{Weight diagrams}

We shall
follow \cite{BS, GS, SZ2012-2} to express  integral dominant weights by weight diagrams.
Given any $\l\in\ZZ^n$, we let
\begin{eqnarray}
\begin{aligned}
&S(\l)\LE=\{\l_i>0\,|\,1\le i\le n\}, &\quad&S(\l)\RI=\{-\l_i>0\,|\,1\le i\le n\},\\
&S(\l)=S(\l)\LE\cup S(\l)\RI,          &\quad& S(\l)\BO=S(\l)\LE\cap S(\l)\RI.
\end{aligned}
\end{eqnarray}
\begin{definition}\label{WeightDiag}\rm
Any element $\l\in\ZZ^n$ can be expressed in a unique way by a {\it weight diagram}
$D_\l$ (cf.~\eqref{Diagram-l}), which is the positive half of the real line with vertices
indexed by $\N=\{0,1,2,...\}$
such that \begin{itemize}
\ITEM
the vertex $i=0$ is associated with $\lfloor\frac{z(\l)}{2}\rfloor$ many $\times$'s,  and  if $\bar z(\l)=1$
an additional symbol $\bot$ at the top;
\ITEM
each vertex $i>0$ is associated with a unique symbol $D_\l^i\in\{\emptyset, \ <, \   >,  \ \times\}$ according to
the following rule: $D_\l^i$ is
\begin{eqnarray*}
\begin{aligned}
\emptyset,& 	\quad \text{ if \ }  i\notin S(\l),\\
<,& 		\quad \text{ if \ }  i\in S(\l)\RI\setminus S(\l)\BO,\\
>,&  		\quad \text{ if \ }  i\in S(\L)\LE\setminus S(\l)\BO, \\
\times,&  	\quad \text{ if \ }  i\in S(\l)\BO.
\end{aligned}
\end{eqnarray*}
\end{itemize}
Numerate the $\times$'s from bottom to top at vertex 0 and then from left to right by $1, 2, \dots, r$ if there are
$r$ of them in total.
\end{definition}

Thus the degree $\#\l$ of atypicality of
$\l$ is the number of $\times$'s in the weight diagram $D_\l$.
For example,
if $\fg=\fq(16)$ and \equa{l==}{\l=(\put(2,59){$\line(0,-1){48}$}
\put(2,59){$\line(1,0){5}\rb{-2pt}{$\,\ssc\g_6\,$}\line(1,0){13}$}7,5,\put(2,50){$\line(0,-1){39}$}
\put(2,50){$\line(1,0){5}\rb{-2pt}{$\,\ssc\g_5\,$}\line(1,0){13}$}4,\put(2,41){$\line(0,-1){30}$}
\put(2,41){$\line(1,0){5}\rb{-2pt}{$\,\ssc\g_4\,$}\line(1,0){13}$}2,\put(2,32){$\line(0,-1){21}$}
\put(2,32){$\line(1,0){5}\rb{-2pt}{$\,\ssc\g_3\,$}\line(1,0){13}$}1,\put(2,24){$\line(0,-1){14}$}
\put(2,24){$\line(1,0){5}\rb{-2pt}{$\,\ssc\g_2\,$}\line(1,0){13}$}0,\put(2,10){$\line(0,1){7}$}
\put(2,17){$\line(1,0){5}\rb{-2pt}{$\,\ssc\g_1\,$}\line(1,0){3}$}\put(25,10){$\line(0,1){7}$}
\put(25,17){$\line(-1,0){4}$}0,0,0,\put(2,24){$\line(0,-1){14}$}
\put(2,24){$\line(-1,0){15}$}0,-\put(2,32){$\line(0,-1){21}$}
\put(2,32){$\line(-1,0){50}$}1,-\put(2,41){$\line(0,-1){30}$}
\put(2,41){$\line(-1,0){85}$}2,-\put(2,50){$\line(0,-1){39}$}
\put(2,50){$\line(-1,0){120}$}4,-\put(2,59){$\line(0,-1){48}$}
\put(2,59){$\line(-1,0){155}$}7,-8,-10),}
with the atypical roots $$\g_1\!=\!\es_7\!-\!\es_9,\ \ \g_2\!=\!\es_6\!-\!\es_{10},\ \ \g_3\!
=\!\es_5\!-\!\es_{11},\ \ \g_4\!=\!\es_4\!-\!\es_{12},\ \ \g_5\!=\!\es_3\!-\!\es_{13},\ \ \g_6\!=\!\es_1\!-\!\es_{14},$$
then the weight diagram $D_\l$ is given \vspace*{-7pt}by
\equa{Diagram-l}{\ \ \ \ \ \ \ \ \ \
\raisebox{-5pt}{$\stackrel{\ssc\bot}{\stackrel{\ssc\!\!\!2\,\times}{\stackrel{\ssc
\!\!\!1\,\times}{\,0\,}}}$}\line(1,0){10} \raisebox{-5pt}{$\stackrel{{\ssc
\!\!\!3\,}\dis
\times}{\,1\,}$}\line(1,0){10} \raisebox{-5pt}{$\stackrel{{\ssc
\!\!\!4\,}\dis
\times}{\,2\,}$}\line(1,0){10}\raisebox{-5pt}{$\,3\,$}
\line(1,0){10}\raisebox{-5pt}{$\stackrel{{\ssc
\!\!\!5\,}\dis \times}{\,4\,}$}
\line(1,0){10}\raisebox{-5pt}{$\stackrel{\dis >}{\,5\,}$}
\line(1,0){10}\raisebox{-5pt}{$\,6\,$}
\line(1,0){10}
%
%
%
\raisebox{-5pt}{$\stackrel{{\ssc
\!\!\!6\,}\dis \times}{\,7\,}$}
\line(1,0){10}\raisebox{-5pt}{$\stackrel{\dis <}{\,8\,}$}
\line(1,0){10}
%
%
\raisebox{-5pt}{$\,9\,$}
\line(1,0){9}\raisebox{-5pt}{$\stackrel{\dis <}{\,10\,}$}
\line(1,0){9}\raisebox{-5pt}{$\,11\,$}\line(1,0){9}\raisebox{-5pt}{$\,12\,$}
\line(1,0){9}\raisebox{-5pt}{$\,13\,$}\line(1,0){9}\raisebox{-5pt}{$\,14\,$}
\line(1,0){9}\raisebox{-5pt}{$\,15\,$}\line(1,0){9} \ .\ .\ .,
}
where, for simplicity, we have left out the symbol $D_\l^i=\emptyset$ from
a vertex $i$ if $i\notin S(\l)$.   In \eqref{Diagram-l},  the number
of $\times$'s is indeed $\#\l=6$.

\subsection{Raising operators}
Given any two vertices $s, t$ with $s\le t$ in a weight diagram $D_\l$,
we define the {\it distance} $d_\l(s,t)$ from $s$ to $t$  to be the
number of $\emptyset$'s minus the number of $\times$'s strictly
between these vertices. Note that the ``distance'' can be negative,  and we remark that $d_\l(s,t)$ is the negative
of the length $\ell_\l(s,t)$ defined in \cite{SZ2012-2}.

Suppose $\#\l=r$, then the
$\times$'s in $D_\l$ are labelled by $1,...,r$.  We denote by $x_i$ the
vertex where the $i$-th $\times$ sits. For convenience, we denote $x_0=0$
(and imagine there is a $0$-th $\times$ sitting on the bottom at vertex $0$).
\begin{definition}\label{defi-length-l}
\rm\begin{itemize}\item[(1)]
Given an $i$ such that $0\le i\le r$ and any vertex $t>0$,
we define the {\it length} $\ell_\l(i,t)$ from the $i$-th $\times$ to the vertex $t$ by
\begin{itemize}
\item[$\bullet$] $\ell_\l(i,t)=d_\l(x_i,t)$ if $x_i>0$, and
\item[$\bullet$] $\ell_\l(i,t)=d_\l(x_i,t)- 2(\lfloor\frac{z(\l)}{2}\rfloor-i)- \sharp(\bot)$  if $x_i=0$,
where $\sharp(\bot)=0$ or $1$ is the number of $\bot$ at vertex $0$.
\end{itemize}
\item[(2)]
A {\it right move}
(or {\it raising operator}) on $D_\l$ is to move to the right a
$\times$, say the $i$-th one ($1\le i\le r$), to the first empty vertex  $t$ (vertex
with the symbol $\emptyset$) that meets the conditions
$\ell_\l(i,t)=0$ and $\ell_\l(i,s)<0$ for all vertices $s$
satisfying $x_i<s<t$. We denote this right move by $R_i(\l)$. We also denote $k_i=t-x_i$, then we obtain
$(k_1,...,k_r)$, which will be referred to as the {\it $r$-tuple of positive integers associated to $\l$}.
\item[(3)]
Given an element $\th=(\th_1,...,\th_r)\in\{0,1\}^r$, we set
$|\th|=\sum_{i=1}^r\th_i$ and let $\theta_{i_1}, \dots, \theta_{i_{|\th|}}$
with $1\le i_1<\cdots<i_{|\th|}\le r$ be the nonzero entries.
Associate to $\theta$ a unique
{\it right path} (or {\it raising operator}) $R_\th(\l)$ which is the collection of the
$|\th|$ right moves $R_{i_1}(\l),...,R_{i_{|\th|}}(\l)$ (see Remark \ref{move-with-l}).

We also use $R_\th(\l)$ to denote the integral dominant weight
corresponding to the weight diagram obtained in the following way.
For each $a=1, \dots, |\theta|$,
let $t_a$ be the vertex where the $i_a$-th
$\times$ of $D_\l$ is moved to by $R_{i_a}(\l)$.
Delete from $D_\l$ all the $\times$'s labeled by
$i_1$, $i_2$, $\dots$, $i_{|\theta|}$, and then place a $\times$
at each of the vertices $t_1$, $t_2$, $\dots$, $t_{|\theta|}$.
\end{itemize}
\end{definition}

As an example, we observe that the third
$\times$ in the weight diagram \eqref{Diagram-l} can only be moved
to vertex $11$, which is the move $R_3(\lambda)$. We indicate all right moves below
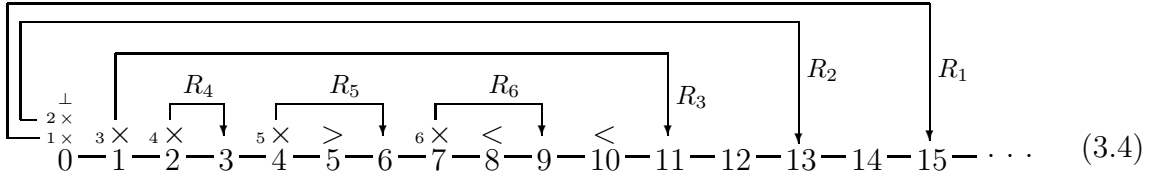
\equa{Diagram-l+}{\hspace*{10pt}\ \ \ \ \ \ \ \ \ \ \ \ \ \ \ %
\put(-17,58){$\line(0,-1){25}\line(1,0){330}$}\put(-17,7){$\line(0,1){30}\line(1,0){12}$}
\put(-12,14){$\line(0,1){20}\line(1,0){6}$}\put(-12,51){$\line(0,-1){29}\line(1,0){260}$}
\raisebox{-5pt}{$\stackrel{\ssc\bot}{\stackrel{\ssc\!\!\!2\,\times}{\stackrel{\ssc
\!\!\!1\,\times}{\,0\,}}}$}\line(1,0){10}
\put(4,39){$\line(0,-1){25}\line(1,0){110}$}
\raisebox{-5pt}{$\stackrel{{\ssc
\!\!\!3\,}\dis
\times}{\,1\,}$}\line(1,0){10}
\put(4,20){$\line(0,-1){6}\line(1,0){20}$\put(-15,4){\footnotesize$R_4$}}
\raisebox{-5pt}{$\stackrel{{\ssc
\!\!\!4\,}\dis
\times}{\,2\,}$}\line(1,0){10}
\put(4,20){$\vector(0,-1){12}\line(-1,0){20}$}
\raisebox{-5pt}{$\,3\,$}
\line(1,0){10}
\put(4,20){$\line(0,-1){6}\line(1,0){20}$\put(0,4){\footnotesize$R_5$}}
\raisebox{-5pt}{$\stackrel{{\ssc
\!\!\!5\,}\dis \times}{\,4\,}$}
\line(1,0){10}\raisebox{-5pt}{$\stackrel{\dis >}{\,5\,}$}
\line(1,0){10}
\put(4,20){$\vector(0,-1){12}\line(-1,0){30}$}
\raisebox{-5pt}{$\,6\,$}
\line(1,0){10}
\put(4,20){$\line(0,-1){6}\line(1,0){20}$\put(0,4){\footnotesize$R_6$}}
\raisebox{-5pt}{$\stackrel{{\ssc
\!\!\!6\,}\dis \times}{\,7\,}$}
\line(1,0){10}\raisebox{-5pt}{$\stackrel{\dis <}{\,8\,}$}
\line(1,0){10}
\put(4,20){$\vector(0,-1){12}\line(-1,0){30}$}
\raisebox{-5pt}{$\,9\,$}
\line(1,0){10}\raisebox{-5pt}{$\stackrel{\dis <}{\,10\,}$}
\line(1,0){9}
\put(7,39){$\vector(0,-1){31}\line(-1,0){120}$}\put(10,20){\footnotesize$R_3$}
\raisebox{-5pt}{$\,11\,$}\line(1,0){9}\raisebox{-5pt}{$\,12\,$}
\line(1,0){9}
\put(7,51){$\vector(0,-1){46}\line(-1,0){50}$}\put(10,30){\footnotesize$R_2$}
\raisebox{-5pt}{$\,13\,$}\line(1,0){9}\raisebox{-5pt}{$\,14\,$}\line(1,0){9}
\put(7,58){$\vector(0,-1){52}$}\put(7,58){$\line(-1,0){120}$}\put(10,30){\footnotesize$R_1$}
\raisebox{-5pt}{$\,15\,$}\line(1,0){9} \ .\ .\ .
}
In particular, if $\th=(1,1,1,1,1,1)$, we have
\[R_\th(\l)=(15,13,11,9,6,5,3,0,-3,-6,-8,-9,-10,-11,-13,-15).\]
In \eqref{Diagram-l+}, we have written $R_j$ for $R_j(\l)$ and will do
this in the future so long as there is no possibility of confusion.

\begin{remark}\label{move-with-l}
We put $\l$ in the notations $R_{i_a}(\l)$ and $L_{i_b,j_b}(\l)$ (to be defined later) to
emphasis the fact that the individual moves in a
right path $R_\theta(\l)$ or left path $L_{\bi,\bj}(\l)$  (to be defined later) are independently
applied to the weight diagram $D_\l$ of $\l$, and {\bf not} to the
resulting diagram of previous moves.
\end{remark}

Note from \cite[Main Theorem]{B1} that
the Euler module $E(\l)$ defined in \eqref{Euler-mod} is in fact a true module if
$\l\in\ZZ^n$. Using the notion of right paths,
\cite[Main Theorem]{B1} can be conveniently restated as follows.
\begin{theorem}\label{compo-theo}\cite[Main Theorem]{B1}
For any $\l,\mu\in\ZZ^n$, let
$a_{\l\mu}:=[E(\l):L(\mu)]$ be
the multiplicity  of the composition factor $L(\mu)$ in $E(\l)$. Then
\begin{eqnarray}
a_{\l\mu}=\left\{\begin{array}{ll}
2^{^{\sc\frac{z(\mu)-z(\l)}2}}&\mbox{if $\l=R_\theta(\mu)$ for some $\theta\in\{0,1\}^{\#\l}$},\\[4pt]
0&\mbox{otherwise}.
\end{array}\right.
\end{eqnarray}
\end{theorem}
\begin{remark}\label{RRR---}If we delete from $D_\l$  all the vertices associateed with the
symbols $<$ and $>$ and relabel vertices,
we obtain an $r$-fold atypical integral dominant weight
$\l_{\rm red}\in\Z_{++}^{2r+\bar z(\l)}$. Then Theorem \ref{compo-theo}
depends only on $\l_{\rm red}$ and $\mu_{\rm red}$, i.e.,  $a_{\l\mu}=a_{\l_{\rm red},\mu_{\rm red}}$.
\end{remark}

\subsection{Lowering operators}
Given any $\mu\in\ZZ^n$,
Theorem \ref{compo-theo} provides a convenient algorithm for determining all $\l\in\ZZ^n$ such that
$L(\mu)$ is a composition factor of the Euler module $E(\l)$.
However, the theorem is unwieldy to use when one wants to determine
the composition factors of a given Euler module $E(\l)$.  We shall derive from
Theorem \ref{compo-theo} an algorithm which is more readily applicable for the latter task.
For this, we need to introduce left moves and left paths following the general ideas of \cite{SZ2012-2}.

\begin{definition}\label{moves-left}
A {\it left move} (or {\it lowering operator}) is to move a $\times$ sitting at a vertex $x_j>0$ to the
left
\begin{itemize}\item[(1)] either to an empty vertex $s>0$ if $d_\l(s,x_j)=0$
(in this case, if the number of $\times$'s strictly
between $s$ and $x_j$ is $k$, we let $i=j-k$ and denote the left
move by $L_{ij}(\l)$);
\item[(2)] or vertex $0$ (and place it below $\bot$ in case there is a
$\bot$ at vertex $0$) if  $d_\l(0,x_j)\in-2\Z_+$ (in this case we denote the left move by $L_{0j}(\l)$).
\end{itemize}
We remark that a left move may move a $\times$ to different places, in contrast to a right move.
\end{definition}
\begin{definition}\label{left-path}
A {\it left path} (or simply a {\it path}) is the collection
of left moves $L_{i_1,j_1}(\l)$, $...$, $L_{i_k,j_k}(\l)$  (cf.~Remark \ref{move-with-l})
satisfying the following conditions
\begin{enumerate}
\item[(i)] $1\le j_1<\cdots<j_k\le r$, and
\item[(ii)] for $1\le a<b\le k$, if $i_b\le j_a$ then $i_b\le i_a$,  and
\item[(iii)] for any $i_b\le p<j_b$, if
$\ell_\l(x_p,x_{j_b})\le0$, then
$p=j_a$ for some $a<b$.
\end{enumerate}
Let $\bi=(i_1,...,i_k)$ and $\bj=(j_1,...,j_k)$, and denote by
$L_{\bi,\bj}(\l)$ the left path. If $k=0$, we use $L_\emptyset$ to
denote this empty path.
We shall also use $L_{\bi,\bj}(\l)$ to
denote the integral dominant weight corresponding to the weight
diagram obtained in the following way. Let $s_a$ be the vertex where
the $j_a$-th $\times$ of $\l$ is moved to by $L_{i_a,j_a}(\l)$ for
$a=1, 2, \dots, k$. Delete from $\l$ the $\times$'s labeled by $j_1,
j_2, \dots, j_k$ and then place a $\times$ at each of the vertices
$s_1, s_2, \dots, s_k$.
Denote \begin{equation}\label{Theta-l}
\Theta^\l= \mbox{ the set of left paths of $\l$.}\end{equation}
We also use $\Theta^\l$ to denote the set of integral dominant weights corresponding to the left paths of $\l$.
\end{definition}

\begin{remark}In Definition \ref{left-path}, \begin{itemize}\ITEM
condition (ii) means that in a left path no two moves are allowed to across each other (like\hspace*{28pt}
    $\put(0,11){$\line(-1,0){25}$}
    \put(-25,11){$\vector(0,-1){10}$}
    \put(0,11){$\line(0,-1){5}$}
    \put(-5,0){$\times$}
    \put(20,0){$\put(0,15){$\line(-1,0){35}$}
    \put(-35,15){$\vector(0,-1){14}$}
    \put(0,15){$\line(0,-1){8}$}\put(-5,0){$\times$}$} $
\hspace*{23pt});
\ITEM condition (iii) means that when a $\times$, say the $b$-th one, is moved to the left passing over another,
say the $p$-th $\times$ (like\hspace*{38pt} %
    $
\put(10,-10){\footnotesize $b$-th}
\put(-15,-10){\footnotesize $p$-th}
%
    \put(-5,0){$\times$}
    \put(10,0){$\put(0,15){$\line(-1,0){45}$}
    \put(-45,15){$\vector(0,-1){14}$}
    \put(0,15){$\line(0,-1){8}$}
    \put(-5,0){$\times$}$} $
\hspace*{22pt}), if the two $\times$'s are too close in the sense that $\ell_\l(x_p,x_{j_b})\le0$,
then the $p$-th $\times$ must also be moved to the left whose destination is
closer than that of the $b$-th $\times$
(like\hspace*{38pt}
    $
\put(10,-10){\footnotesize $b$-th}
\put(-15,-10){\footnotesize $p$-th}
    \put(0,11){$\line(-1,0){25}$}
    \put(-25,11){$\vector(0,-1){10}$}
    \put(0,11){$\line(0,-1){5}$}
    \put(-5,0){$\times$}
    \put(10,0){$\put(0,15){$\line(-1,0){45}$}
    \put(-45,15){$\vector(0,-1){14}$}
    \put(0,15){$\line(0,-1){8}$}
    \put(-5,0){$\times$}$} $
\hspace*{22pt}).
\end{itemize}
\end{remark}

We give some examples of left paths to illustrate the concept.
\begin{example}
If $\l$ is the weight in \eqref{Diagram-l}, one can easily obtain
all the possible left paths for $D_\l$. There are only 4 paths in total,
which are given by (cf.~Remark \ref{move-with-l})
\equa{left-paths-ex}{\begin{array}{lllllllllllll}L_\emptyset&&
L_{55}&&L_{66}&&L_{56}.
\end{array}}
\end{example}

\begin{example}
Let $n=2r$ and $\l=(2r-1,...,3,1,-1,-3,...,-2r+1)\in\ZZ^n$. Then for any $i,j$ with $0\le j\le i\le r$,
$j\ne1$ and $i\ge1$, we have
the left move $L_{ji}$ (note that in this example we do not have $L_{1i}$ as by
Definition \ref{moves-left}(1), $L_{ji}$ for $j\ne0$  is to move to the left the $i$-th $\times$
to some empty vertex $s>0$; if we can allow $s=0$, then in fact $L_{1i}$
coincides with $L_{0i}$; this is why we require $s>0$). The total number of left paths is
\begin{equation}\label{2r-D=l}
\#\Theta^\l=(r+1)C_{r}=\binom{2r}{r},
\end{equation}
 where $C_r:=\frac1{r+1}\binom{2r}{r}$ is the {\it $r$-th Calatan number}.

 To prove \eqref{2r-D=l}, let $\#\Theta^\l=x_r$. Clearly $x_0=1,\,x_1=2$.
Assume $r\ge2$. Considering the $r$-th $\times$, we have the following choices:
(1) it is not moved; (2) $L_{rr}$, (3) $L_{0r}$ and (4) $L_{ir}$ for $2\le i\le r-1$.

Note that for the first 3 choices, the first $(r-1)$ $\times$'s can be moved to the left
to any places so long as conditions in Definition \ref{left-path} are satisfied by these
$(r-1)$ $\times$'s. Thus there are in total $3x_{r-1}$ left paths of this kind.

For choice (4), when $i$ is fixed, we  have the following:
\begin{itemize}\ITEM the first $(i-1)$ $\times$'s can be moved to the left to any places so long as conditions in
Definition \ref{left-path} are satisfied by these  $(i-1)$ $\times$'s. Thus we have $x_{i-1}$ choices.
\ITEM The $i$-th $\times$ cannot be moved to the left by condition (ii) in Definition \ref{left-path}.
\ITEM The remaining $(r-i-1)$ $\times$'s can be moved to the left to any places not passing over the $i$-th $\times$,
so long as conditions in Definition \ref{left-path} are satisfied by these $\times$'s.
Since vertex $0$ is not involved in such left moves, in this case,
conditions in Definition \ref{left-path} are exactly the same as that of left paths for the general linear
superalgebra $\gl(m|n)$
 \cite[Definition 5.2]{SZ2012-2}.
Thus we have $C_{r-i}$ choices by \cite[Lemma 3.12]{Su2006}.
\end{itemize}
Therefore, we obtain
\begin{equation}\label{C-----sss}x_r=3x_{r-1}+\sum_{i=2}^{r-1}x_{i-1}C_{r-i}\mbox{ for }r\ge2.\end{equation}
 From this, one can deduce \eqref{2r-D=l} by induction on $r$.
\end{example}

\begin{example}
Let $n=2r+1$ and $\l=(2r,...,4,2,0,-2,-4,...,-2r)\in\ZZ^n$. Then for any $i,j$ with $0\le j\le i\le r$ and $i\ge1$, we have
the left move $L_{ji}$ (note that in this example, we have $L_{1i}$ which is different from $L_{0i}$,
in contrast to the previous example). The total number of left paths is\begin{equation}\label{2r+1-D=l}\#\Theta^\l=\frac12(r+2)C_{r+1}=\frac12\binom{2r+2}{r+1}.\end{equation}
The proof is the same as that of the previous example with \eqref{C-----sss} replaced by
$x_r=3x_{r-1}+\sum_{i=1}^{r-1}x_{i-1}C_{r-i}$ for $r\ge1$. From this we obtain \eqref{2r+1-D=l}.
\end{example}

\subsection{Structure of Euler modules}

The notion of left paths provides a useful way to describe the structure of Euler modules. We have the following result.
\begin{theorem}\label{composi-f}
Let $\l\in\ZZ^n$ be given. Then $L(\mu)$ with $\mu\in\ZZ^n$ is a composition factor of $E(\l)$ if and only if
$\mu\in\Theta^\l$. In this case, the multiplicity is \raisebox{-1pt}{$a_{\l\mu}=2^{^{\sc\frac{z(\mu)-z(\l)}2}}$}.
In particular, the number of composition factors of $E(\l)$ $($without
counting the multiplicities$)$ can vary from  $1$ to the maximal number $\frac12\binom{2r+2}{r+1}$,
where $r=\#\l$ is the degree of atypicality of $\l$.
\end{theorem}
\begin{proof}
The first and second statements are a reformulation of Theorem \ref{compo-theo} in terms of left paths (cf. Definition \ref{left-path}), while the third statement follows from the first.
\end{proof}

\section{Character and dimension  formulae}
\subsection{The $c$-relationship}

Following \cite{HKV,SZ2007}  we introduce the notion of $c$-relation
between atypical roots of an integral dominant weight $\l$.

\begin{definition}\label{c-related}
Let $\l$ be an $r$-fold atypical integral dominant weight with atypical roots \eqref{aty-llll}.
Suppose $1\le s\le t\le r$.
\begin{itemize}\item[(1)]
 We define the {\it distance } $d_{st}(\l)$ of
 two atypical roots $\g_s,\g_t$ of $\l$ by $d_{st}(\l)=0$
  if the $s$-th and $t$-th $\times$'s both sit at vertex $0$, and
 $d_{st}(\l)=\ell_\l(s,x_t)$ otherwise, where $x_t$ is the vertex where the $t$-th $\times$ sits and $\ell_\l(s,x_t)$ is
defined in Definition \ref{defi-length-l}.
\item[(2)]
We say that two atypical roots $\g_s,\g_t$ of $\l$
are {\it $c$-related} \cite{HKV,SZ2007} (or {\it connected} \cite{VZ}) if
 $d_{st}(\l)\le 0$, and are {\it strongly $c$-related} if
$\g_s,\g_p$ are $c$-related for all $p$ such that $s\le p\le t$.
\item[(3)] As in \cite{SZ2007}, we let
$\hat c_{st}(\l)=1$ if  $\g_s,\g_t$ are strongly $c$-related, and $\hat c_{st}(\l)=0$ otherwise.
\end{itemize}
\end{definition}

\begin{example}
Let $\l$ be as in \eqref{l==} with weight diagram \eqref{Diagram-l} and
right moves \eqref{Diagram-l+}. One immediately sees that $\hat c_{st}(\l)=1$
for $s=1,2,3$ and $s\le t\le 6$, and $\hat c_{st}(\l)=0$ for $s=4,5,6$
and $s<t\le6$. In general if $1\le s<t\le r$, $\hat c_{st}(\l)=1$ if and
only the right move $\tilde R_s(\l)$ passes over the $t$-th $\times$,
where $\tilde R_s=R_s$ if $x_s\ne0$ or there is a $\bot$ at vertex $0$,
otherwise it is the move further to the right to the first empty vertex
that is available (just image there is a $\bot$ at vertex $0$,
then  $\tilde R_s$ is the move $R_s$).
\end{example}

We need to understand $\hat c_{st}(\l)$ better. Recall that in the weight diagram $D_\l$ of $\l$,
there are $\lfloor \frac{z(\l)}{2}\rfloor$ many $\times$'s
(and in addition one $\bot$ if $\bar z(\l)=1$) located at vertex $0$. The following result can be easily verified.
\begin{lemma}\label{c-st-==}
Let $\l$ be an $r$-fold atypical integral dominant weight as in \eqref{(lambda-1)} with atypical roots
 $\g_p=\es_{m_p}-\es_{n_p}$ for $1\le p\le r$ as in \eqref{aty-llll}. Then for $1\le s<t\le r$,
\equa{hat-c-st}{
\hat c_{st}(\l){\sc}={\sc}\left\{{\sc}{\sc}\begin{array}{ll}
1{\sc}{\sc}{\sc}&\mbox{if }d_{sp}{\sc}\le{\sc}0\mbox{ for all $p$ with }s{\sc}<{\sc}p{\sc}\le{\sc} t,\\[4pt]
0&\mbox{otherwise},
\end{array}\right.
}
where $d_{sp}:=\l_{m_p}{\sc\!}-{\sc\!}\l_{m_s}{\sc\!}+m_p{\sc\!}-{\sc\!}m_s
{\sc\!}+n_s{\sc\!}-{\sc\!}n_p{\sc\!}+{\sc\!}1{\sc\!}-{\sc\!}d_s$, and
$d_s=\bar{z}(\l)$ if $s\le\lfloor \frac{z(\l)}{2}\rfloor$ or $0$ else.
\end{lemma}

\subsection{Lexical weights}
Let $\L$ be an $r$-fold atypical regular weight
(not necessarily dominant)
with the set
$\G_\L=\{\g_1,...,\g_r\}$
of atypical roots
ordered according to (\ref{aty-llll}). As in \cite{SZ2007}, we define the {\it atypical tuple of $\L$}
\equa{tt-f}{\tt_\L:=(\l_{m_r},...,\l_{m_1})\in\Z^r,}
and call $\l_{m_s}$ the {\it $s$-th atypical entry of $\L$} for $1\!\le\!s\!\le\!r$.
We also define
\equa{ot-f}
{
\ot_\L\in\Z^{n-2r}\mbox{ \ and \ }\bar\l\in\Z^n
}
to be respectively the element obtained
from $\L$ by deleting all entries $\L_{\mm_s},\,\L_{\cp\nn_s}$ and the element obtained by changing
 all entries $\L_{\mm_s},\,\L_{\cp\nn_s}$ to zero
for $s=1,...,r$. We call $\ot_\L$ the {\it typical tuple of $\L$}, whose all entries,
called the {\it typical entries of $\L$}, have distinct absolute values by (\ref{l-ru}).

\def\PP{{\mathcal P}}In the following, we fix an $r$-fold atypical integral dominant weight $\l$.
 Denote by $\PP^\l$ the set of all integral weights $\mu$
such that $\G_\l$ is a maximal set of orthogonal atypical roots of $\mu$ (with respect to the bilinear form
\eqref{form-}) and $\bar\mu=\bar\l$. Hereafter,
 $\tt_\mu,\ot_\mu,\bar\mu$ are defined as in \eqref{tt-f} and \eqref{ot-f}. Thus any $\mu\in\PP^\l$ has the form
 \equa{mu-form}{\mu=\bar\l+\sum_{p=1}^r j_p\g_p\mbox{ \ for some \ }j_p\in\Z.}
If $\nu=\bar\l+\sum_{p=1}^r\ell_p\g_p\in\PP^\l$, we define  (cf.~\eqref{regular-def})
\equa{level---}{|\mu-\nu|=|\mu^+-\nu^+|=\sum_{p=1}^r(j_p-\ell_p)
\mbox{ \ (called the {\it relative level} of $\nu$ in $\mu$)}.}

\begin{definition}\label{defi-lexical-P-l}
\begin{itemize}\item[(1)]
We call $\mu\in {\mathcal P}^\l$ {\it lexical} if its atypical tuple $\tt_\mu$ is lexical, where
an element $a=(a_r,a_{r-1},...,a_1)\in\Z^r$ is called {\it lexical}
if
\equa{lexical}
{
a_r\ge a_{r-1}\ge ...\ge a_1.
}
\item[(2)]
Denote by $\Dr$ the subset of $\PP^\l$ consisting of the lexical weights of $\PP^\l$.
\end{itemize}
\end{definition}

We define the partial
order $\soe$ on $\PP^\l$ (which is compatible with the total order $<$ defined in \eqref{partial-order}) for $\mu,\nu\!\in\! \PP^\l$ by
\equa{g-prec-f} { \nu\soe\mu\ \ \ \Lra\ \ \ \tt_\nu\le
\tt_\mu,}
where the partial order
``$\le$'' on $\Z^r$ is defined for
$a\!=\!(a_1,...,a_r),\,b\!=\!(b_1,...,b_r) \!\in\!\Z^r$ by
\equa{partial-order-on-Z-r} { a\le b\ \ \ \Lra\ \ \ a_i\le b_i\ \
\mbox{for \ }1\le i\le r. }
 For any $\mu\in \PP^\l$, we denote $\PP^{\soe\mu}:=\{\nu\in\Dr\,|\,\nu\soe\mu\mbox{ \ and \ }\tt_\nu\in\N^r\}$, then
\begin{eqnarray}\label{P-mu==}
\PP^{\soe\mu}
&\!\!\!=\!\!\!&\!
\left.\left\{\!\bar\mu\!+\!\sum_{p=1}^r j_p\g_p\,\right|\,0\!\le\! j_1\!\le\! m_1
\mbox{ and }j_{p-1}\!\le\! j_p\!\le\! m_p\mbox{ for }p\!=\!2,...,r \right\},
\end{eqnarray}
where $m_p={\rm min}\{\mu_p,\mu_{p+1},...,\mu_r\}.$

\subsection{The action of $\Sr_r$ on $\PP^\l$}

The symmetric group $\Sr_r$ of degree $r$ acts on $\Z^r$ by
permuting entries. This action induces an action on $\PP^\l$ given by
{\def\L{\mu}\equa{(4)}
{{\sc\!\!\!\!\!\!\!\!\!} \si(\L)\!=\!(\L_1,{\sc...\,},
\put(4,-1){$\line(0,-1){7}$} \put(-25,-17){\small atypical entries
permuted} \L_{\mm_{\si^{-1}(r)}},{\sc...\,},\L_i,{\sc...\,},
\put(4,-1){$\line(0,-1){7}$}
\L_{\mm_{\si^{-1}(1)}},{\sc...\,},
\put(4,-1){$\line(0,-1){7}$} \put(-15,-17){\small atypical entries
permuted} \L_{\cp
\nn_{\si^{-1}(1)}},{\sc...\,},\L_{\cp\zeta},{\sc...\,},
\put(4,-1){$\line(0,-1){7}$} \L_{\cp
\nn_{\si^{-1}(r)}},{\sc...\,},\L_{\cp n}),\!\!\!\!\!\!\! }
}%
for
$\si\in \Sr_r$ and $\mu\in \PP^\l$. With this action on $\PP^\l$, the
group $\Sr_r$ can be regarded as a subgroup of $W$, such that
every element is of even parity.
We will need the following
\def\Sl{{\mathcal S}^\l}
\begin{definition}\label{a-l-si}
Let $\mu\in \PP^\l,\,\si\in\Sr_r$, we define
\equa{aaa-lll}{a_\mu^\si=\left\{\begin{array}{ll}0&\mbox{if $\hat c_{st}(\mu)\!=\!1$,
$s\!<\!t$, $\si^{-1}(s)\!>\!\si^{-1}(t)$ for some $s,t$,}\\[4pt]1&\mbox{otherwise},\end{array}\right.}
where $\hat c_{st}(\mu)$ is defined in \eqref{hat-c-st}. Define
\equa{S-llll}{{\mathcal S}^\mu=\{\si\in\Sr_r\,|\,a_\mu^\si=1\}.}
Namely, ${\mathcal S}^\mu$ is the subset of the symmetric group $\Sr_r$ consisting
of permutations $\si$ which do not change the order of $s <t$ when the atypical roots $\g_s$ and
$\g_t$ of $\mu$ are strongly $c$-related.
\end{definition}

\subsection{More on raising operators}First we generalise the notion of $r$-tuples of positive integers
associated to weights as follows.
\begin{definition}\label{reu-tuple}
Let $\mu$ be a regular weight (cf.~Definition \ref{regular-def}(1)).
Assume that $\mu$ is $r$-fold atypical with atypical roots $\g_1<\g_2<\cdots<\g_r$ as in \eqref{aty-llll}.
Let $w\in W$ be such that $\mu^+:=w(\l)$ is integral dominant. Then $\G_{\mu^+}=w(\G_\mu)$.
Arrange the $r$ atypical roots of $\mu^+$ as \equa{l+++}
{w(\g_{i_1})<w(\g_{i_2})<\cdots<w(\g_{i_r}).}
We define the {\it $r$-tuple of positive integers associated to $\mu$}
to be the $r$-tuple $(k_1,...,k_r)$ such that $(k_{i_1},...,k_{i_r})$ is
the $r$-tuple of positive integers associated to $\mu^+$.
\end{definition}

For each $i$ ($1\le i\le r$), we define a {\it raising operator}
$\bar R_i$ on the regular weight $\l$ by
\begin{equation}\label{Ano-Rais}
\bar R_i(\l)=\l+k_i\g_i.\end{equation}
In particular,  when $\l$ is integral dominant, by Definition \ref{defi-length-l}, we have
$R_i(\l)=(\bar R_i(\l))^+$ (recall from Definition \ref{regular-def}(1) that in general $\mu^+$ is
the ``dominant conjugate'' of $\mu$)
and
\begin{equation}\label{Rassss}
R_\theta(\l)=(\bar R_r^{\theta_r} \bar R_{r-1}^{\theta_{r-1}}\cdots\bar R^{\theta_1}_1(\l))^+
=R_{j_r}^{\theta_r}  R_{j_{r-1}}^{\theta_{r-1}}\cdots R^{\theta_1}_{j_1}(\l)
\mbox{ \ for }\theta\in\{0,1\}^r,\end{equation}
where
$j_p=p-\#\{q<p\,|\,\theta_q=1,\,\hat c_{qp}(\l)=1\}$, and the product of operators $R_i$ is their composition.
For example,  $R_iR_j(\l)=R_i(R_j(\l))$, which means applying first  the right move of the $j$-th $\times$ of $\l$ then
the right move of the $i$-th $\times$ to the resulting diagram.
Now for any $\theta=(\theta_1,...,\theta_r)\in\N^r$,  we let
$\bar R_{\theta}=\bar R_1^{\theta_1} \bar R_2^{\theta_2} \cdots  \bar R_r^{\theta_r}$
and define the operator $R'_\theta$ acting on integral dominant weights by \begin{equation}\label{Rainddddd}
R'_\theta(\l)=(\bar R_{\theta}(\l))^+
=(\bar R_1^{\theta_1} \bar R_2^{\theta_2} \cdots  \bar R_r^{\theta_r}(\l))^+
=R_1^{\theta_1} R_2^{\theta_2} \cdots  R_r^{\theta_r}(\l).
\end{equation}

\subsection{Character and dimension formulae}

For any $\l,\mu\in\ZZ^n$, we define
\begin{equation}\label{b-l-mu}
b_{\l\mu}=2^{^{\sc\frac{z(\mu)-z(\l)}2}}
\sum_{\theta\in\N^r:\,\l=R'_\theta(\mu)}(-1)^{|\theta|},\quad \mbox{ where }|\theta|=\sum_{i=1}^r\theta_i.
\end{equation}

\begin{theorem} Order the elements of $\ZZ^n$ by \eqref{partial-order},  and let $A=(a_{\l\mu})_{\l,\mu\in\ZZ^n}$,
$B=(b_{\l\mu})_{\l,\mu\in\ZZ^n}$, which are upper triangular matrices with diagonal
entries being $1$. Then $B=A^{-1}$, and it follows that
\equa{1-har}{{\rm ch\,}L(\l)=\sum_{\mu\in\ZZ^n}b_{\l\mu}{\rm ch\,}E(\mu)
=\sum_{\mu\in\ZZ^n}b_{\l\mu}2^{\lfloor\frac{h(\mu)+1}{2}\rfloor}P_\mu.}
\end{theorem}
\begin{proof}
It suffices to prove  that for $\mu<\l$ (cf.~\eqref{partial-order} or \eqref{g-prec-f}),
\begin{equation}\label{to-prove-a-b-}
\sum_{\nu\in\ZZ^n}b_{\l\nu}a_{\nu\mu}=0.
\end{equation}
Let ${\Omega^\l_\mu}=\{(\theta',\theta)\in\N^r\times\{0,1\}^r\,|\,\l=R'_{\theta'}R_\theta(\mu)\}$.
Then clearly,
\begin{equation}\label{to-prove-a-b-1}
\mbox{left-hand side of \eqref{to-prove-a-b-}}\ \ =\ \ 2^{^{\sc\frac{z(\mu)-z(\l)}2}}\sum_{(\theta',\theta)\in{\Omega^\l_\mu}}(-1)^{\theta'}.
\end{equation}
We shall define a bijective map $\widetilde{\,\ }:{\Omega^\l_\mu}\to{\Omega^\l_\mu},$
$(\theta',\theta)\mapsto(\tilde\theta',\tilde\theta)$ satisfying \equa{bijection-tilde}
{\widetilde{\widetilde{\,\ }}={\rm id}_{{\Omega^\l_\mu}},
\ \ \ |\tilde\theta'|=|\theta'|\pm1,\,\ \ |\tilde\theta|=|\theta|\mp1.}
For $(\theta',\theta)\in{\Omega^\l_\mu}$,
let \equa{p-q-pp}{p={\rm max}\{p\,|\,\theta_p\ne0\},\ \ \ q={\rm max}\{q\ge p\,|\,\hat c_{pq}(\mu)=0\},
\ \ \ p'={\rm max}\{p'\,|\,\theta'_{p'}\ne0\}.}
(We take $p=q=0$ if $\theta=0$, and $p'=0$ if $\theta'=0$.)
If $\theta=0$ (then $\theta'\ne0$ as $\l\ne\mu$) or $q<p'$ (then by definition of $p$ and $q$,
we have $\theta_i=0$ for $i>q$), we set $\tilde\theta'=\theta'-\epsilon_{p'}\in\N^r,\,
\tilde\theta=\theta+\epsilon_{p'}\in\{0,1\}^r$, where in general $\epsilon_i=(\d_{1i},...,\d_{ri})$.
Then $\l=R'_{\tilde\theta'}R_{\tilde\theta}(\mu)$ by \eqref{Rassss} and \eqref{Rainddddd}, i.e., $(\tilde\theta',\tilde\theta)\in\Omega_\mu^\l$.

Assume $\theta\ne0$ and $p'\le q$. We denote
\begin{eqnarray}\label{thera====}
&\!\!\!\!\!\!\!\!\!\!&
\hat\theta'=(\theta'_1,...,\theta'_{p-1},0,...,0),\ \ \bar\theta'=(0,...,0,\theta'_{p},...,\theta'_{q-1},0,...,0)
,\nonumber\\[-4pt]
&\!\!\!\!\!\!\!\!\!\!&
\hat\theta=(\theta_1,...,\theta_{p-1},0,...,0),\ \ \, \, \bar\theta=(\,\put(10,16){$\ssc p-1$}
{\overbrace{0,...,0}},\theta_{p+1},...,\theta_{q},0,...,0)
.\end{eqnarray}
Note from \eqref{Rassss} that $R_\theta(\mu)=R_{\bar\theta}R_pR_{\hat\theta}(\mu)$,
since after applying the right move $R_p$ to $R_{\hat\theta}(\mu)$, the $i$-th $\times$ for $p<i\le q$ becomes
the $(i-1)$-th $\times$ in the resulting diagram.
Then
by \eqref{Rainddddd} and definitions of $p,q,p'$ in \eqref{p-q-pp}, we have
\begin{eqnarray}
\l&\!\!\!=\!\!\!&R'_{\theta'}R_\theta(\mu)=R'_{\hat\theta'}R'_{\bar\theta'}R_{q}^{\theta'_q}
R_{\bar\theta}R_pR_{\hat\theta}(\mu)
=R'_{\hat\theta'}R'_{\bar\theta'}R_{\bar\theta}R_{p}^{\theta'_q+1}R_{\hat\theta}(\mu).
\end{eqnarray}
Denote $\hat\mu=R_{\bar\theta}R_{p}^{\theta'_q+1}R_{\hat\theta}(\mu),$
 $\bar\l=R'_{\bar\theta'}(R_{\bar\theta}(\hat\mu))$. Then $\l=R'_{\hat\theta'}(\bar\l)$.
If $\bar\l=\hat\mu$, i.e., $\bar\theta'=\bar\theta=0$,
then we set $\tilde\theta'=\theta'+\epsilon_p\in\N^r,\,\tilde\theta=\theta-\epsilon_p\in\{0,1\}^r$, and
$(\tilde\theta',\tilde\theta)\in\Omega_\mu^\l$.

Now assume $\hat\mu<\bar\l$. Then $(\bar\theta',\bar\theta)\in\Omega_{\hat\mu}^{\bar\l}$.
By induction on the relative level $|\l-\mu|$ (cf.~\eqref{level---}), we can
assume there exists a bijective map $\widetilde{\,\ }$ on $\Omega^{\bar\l}_{\hat\mu}$
satisfying \eqref{bijection-tilde}.
Now we take $\tilde\theta'=\hat\theta'+\tilde{\bar\theta}'$,
$\tilde\theta=\hat\theta+\epsilon_p+\tau(\tilde{\bar\theta})$,
where  $\tau(\tilde{\bar\theta})$ is obtained from
$\tilde{\bar\theta}$ moving the $0$ at the last entry to the first entry
 (note that $\theta=\hat\theta+\epsilon_p+\tau(\bar\theta)$ by \eqref{thera====}).

The above uniquely defines the bijection $\widetilde{\,\ }$ satisfying \eqref{bijection-tilde},
from this we see that the right-hand of \eqref{to-prove-a-b-1} is zero. Thus we have
\eqref{to-prove-a-b-}.
\end{proof}

For any $\mu\in \PP^\l$, we generalise the definition of
Schur's P-function \eqref{Schur-P} by defining
\equa{g-Schur-P}{\mbox{$\dis P_\mu:=\frac1{\#S_\mu}\sum\limits_{w\in \Sr_n}w
\left(e^\mu\prod\limits_{\stackrel{\ssc 1\le i<j\le n}{\ssc (\mu_i,\mu_j)\ne(0,0)}}
\frac{1+e^{\es_j-\es_i}}{1-e^{\es_j-\es_i}}\right),$}}
where $S_\mu$ is the stabilizer of $\l$ in $\Sr_n$.
Note from Definition \ref{regular-def} that \equa{Notzero}{P_\mu\ne0\ \Lra\ \mu\mbox{ \ is regular.}}
Also recall the definition of $\PP^{\soe\nu}$ from \eqref{P-mu==}.  We have the following result.
\begin{theorem}\label{main-theo}
Let $L(\L)$ be the finite dimensional simple $\fq(n)$-module with highest weight $\L$.
\begin{itemize}\item[\rm(1)]
The formal character of $L(\L)$ is given by
\equa{char}{\dis
\ch L(\l)
=\sum_{\si\in\Sr_r}
\sum_{\mu\in\PP^{\soe\si(\l)}} (-1)^{|\L-\mu|}2^{\frac{z(\mu)-z(\L)}{2}
+\left\lfloor\frac{h(\mu)+1}{2}\right\rfloor}a^\si_\L  P_\mu,
}
where $a^\si_\L$ is defined by \eqref{aaa-lll}.
\item[\rm(2)]
The dimension of $L(\L)$ is given by
\begin{eqnarray}\label{dim}
{\rm dim\,} L(\L)
&=&\sum_{\sc\si\in\Sr_r}\sum_{\mu\in\PP^{\soe\si(\l)}}
\sum_{B\subset\Phi_{\bar0}^+}
(-1)^{|\l-\mu|}2^{\frac{z(\mu)-z(\L)}{2}+\left\lfloor\frac{h(\mu)+1}{2}\right\rfloor}
\frac{a^\si_\L}{\#S_\mu}\nonumber\\
& &\phantom{space-}
\times (-1)^{|B\cap\Phi_{\bar0}^+(\mu)|}
\prod_{\a\in\Phi_{\bar0}^+}\frac{(\a,\mu+\sum_{\b\in B}\b+\rho_{\bar0})}{(\a,\rho_{\bar0})},
\end{eqnarray}
where $\Phi_{\bar0}^+(\mu)\!=\!\{\es_i\!-\!\es_j\!\in\!\Phi_{\bar0}^+\,|\,i\!<\!j,\,\mu_i\!=\!\mu_j\!=\!0\}$.
\end{itemize}
\end{theorem}
\begin{proof}
Part (2) of the theorem can be proven in a similar way as in \cite{SZ2007},
thus we only give the details of the proof for part (1).
By \eqref{S-llll}, we can assume $\si$ appearing in \eqref{char} is in $\Sl$.
By \eqref{Notzero}, we can assume $\mu$ appearing in \eqref{char}
is regular and lexical.
For convenience, without loss of generality, we may assume that the weight diagram $D_\l$ of $\l$
does not contain symbols $>$ and $<$ (i.e., $\l\in\Z_{++}^{2r+\bar z(\l)}$ with
$\tt_\l$ being empty if $\bar z(\l)=0$ or zero else, cf.~Remark \ref{RRR---}).
In this case, regular and lexical weights coincide with integral dominant weights.
Let $\mu\soe\l$ be fixed. We want to prove that \equa{Theta===}
{\Theta^\l_\mu:=\{\theta\in\N^r\,|\,R'_\theta(\mu)=\l\},} is in 1-1 correspondence with $\Sl$. If so,
then \eqref{char} follows from the claim and  \eqref{b-l-mu}, by noting that
$(-1)^{|\theta|}=(-1)^{|\l-\mu|}$ for $\theta\in\Theta^\l_\mu$ (which can be easily proven).

Let $\theta\in\Theta^\l_\mu$, we define the unique element $\si_\theta\in\Sr_r$ such that for each
$p=1,...,r$, the $p$-th $\times$ of $D_\mu$ is moved to the $\si_\theta(p)$-th $\times$ of
$D_\l$ by the raising operator $R'_\theta$.
Note that if $1\le s<t\le r$ such that $\bar c_{st}(\l)=1$, then the $t$-th $\times$ of
$D_\l$ has to be reached by $\si^{-1}_\theta(t)$-th $\times$ of $D_\mu$
 earlier than the $s$-th $\times$ of $D_\l$
 by $\si^{-1}_\theta(s)$-th $\times$ of $D_\mu$
by the rasing operator
 $R'_\theta$ (for example, assume $\l$ is as in \eqref{Diagram-l+}, if the $3$-rd $\times$ of $D_\l$ is
obtained by some raising operator from some $\times$ of some $D_\mu$ before the $4$-th $\times$,
then there is no way that the $4$-th $\times$ of $D_\l$ can be obtained by some rasing operator
from any $\times$ of any $D_\mu$). Thus by definition of $R'_\theta$ in \eqref{Rainddddd},
we have $\si^{-1}_\theta(t)>\si^{-1}_\theta(s)$. This proves $\si_\theta\in\Sl$.
Now one can easily see that the map $\Theta^\l_\mu\to\Sl:\theta\mapsto\si_\theta$ is a bijection.
This proves the claim and the theorem.\end{proof}

\subsection{A corollary}
We can manipulate Theorem \ref{main-theo}  to express ${\rm ch\,}L(\l)$ and ${\rm dim\,}L(\l)$ in more convenient forms.
For any $r$-fold atypical integral dominant weight
$\l$  with atypical roots $\g_1<...<\g_r$ as in \eqref{aty-llll},
we define the {\it normal cone with vertex $\l$} (cf.~\eqref{tt-f}, \eqref{ot-f})  in analogy to \cite{SZ2007}:
\begin{eqnarray}\label{(13)}
&\dis\NC\l\!\!\!&\dis=
\left.\left\{\mu:=\l-\sum\limits_{s=1}^r i_s\g_s\,\right|\,i_s\ge0,\,\tt_\mu\in\N^r\right\}\nonumber\\&&=
\dis\left.\left\{\!\bar\l\!+\!\sum_{p=1}^r j_p\g_p\,\right|\,0\!\le\! j_p\!\le\! \l_{m_p}\mbox{ for }p\!=\!1,...,r \right\}.
\end{eqnarray}
We also define $\TC\l$,
which was referred to as the {\it truncated cone with vertex $\l$} in \cite{VHKT},
to be the subset of $\NC\l$
consisting of weights $\mu$ such that
the $s$-th entry of the atypical tuple $\tt_\mu$
is smaller than or equal to
the $t$-th entry of $\tt_\mu$
when the atypical roots $\g_s,\g_t$ of $\l$ (not $\mu$)
are strongly $c$-related
for $s<t$. Namely
\equa{cf1}{
\begin{array}{c}
\mbox{$
\TC\l=\{\mu\in\NC\l\,|\,
\mu_{\cp \mm_s}\le
\mu_{\cp \mm_t}\mbox{ \ if \ }\dd^\l_{s,t}=1\mbox{ \ for \ }s<t\}.
$}
\end{array}
}

\def\L{\lambda}\def\l{\mu}

For each $\mu\in\TC\lambda$, we define $b^\L_\mu\in\C$ as follows.
We set $b^\L_\mu=0$ if $\mu$ is not regular and $b^\L_\mu=b^{\L_{\rm red}}_{\mu_{\rm red}}$ otherwise,
where $\L_{\rm red}$ is defined as in
 Remark \ref{RRR---} and $\mu_{\rm red}$ is defined as follows: First take $\si\in\Sr_n$ such that
 $\nu:=\si(\mu)$ is dominant. Then we have $\nu_{\rm red}$.
Take $\mu_{\rm red}=\si^{-1}(\nu_{\rm red})$ (note that $\si^{-1}$ induces an action on atypical
entries of $\mu$, thus induces an action on $\nu_{\rm red}$).
Thus we can suppose $\L=\L_{\rm red}$ and $\mu=\mu_{\rm red}$. Note that in this case, the $p$-th atypical root of $\l$ is $\g_p=\es_{r-p+1}-\es_{n-r+p}$, and $\L_p=\L_{m_{r-p+1}}$. If $\mu=0$, we define (cf.~\eqref{Theta===})
\equa{SSSSSS}{b^\L_0:=\#\Theta^\lambda_0
=\frac{\#{\mathcal S}^\L}{(r+\bar z(\L))!} d_{\L_{1}+s_1,r}d_{\L_{2}+s_2,r-1}\cdots d_{\L_{r}+s_r,1},
}
where $d_{i,k}={\rm min}\left\{\left\lfloor\frac{i+1-\bar z(\L)}{2}\right\rfloor,k\right\}$,
$s_p$ is the number of $q<p$ such that $\hat c_{pq}(\L)=1$, and where the last equality follows from combinatorics.
Suppose $\mu\not=0$. Set $i>0$ to be the largest such that $\mu_i>0$.
Let $\mu'\in\Z^{n-2}$ be obtained from $\mu$ by deleting its $i$-th and $(m+1-i)$-th entries,
and let $\L'\in\Z_+^{n-2}$ be obtained from $\L$  by deleting its $i$-th and $(n+1-i)$-th entries and
subtracting its $j$-th entry by $2$ and adding its $(n+1-j)$-th entry by $2$ for all $j<i$.
Then $b^\L_\mu=b^{\L'}_{\mu'}.$
This defines all $b^\L_\mu$.

Note that $b^\lambda_\mu$ can also be described in the following way. For any fixed $\lambda$, each regular
$\mu\in\TC\lambda$ uniquely corresponds to a weight diagram $D_\mu$ which is simply the
diagram $D_{\mu^+}$ (cf.~\eqref{regular-def}) but with the $i_p$-th $\times$ of $D_{\mu^+}$
relabeled as the $p$-th $\times$ of $D_\mu$ (where the $p$-th atypical root of $\mu$ corresponds
to the $i_p$-th atypical root of $\mu^+$, cf.~\eqref{l+++}
in Definition \ref{reu-tuple}) for each $p=1,...,r$. Now similar to Remark \ref{RRR---},
delete all vertices of $D_\lambda$ which associate with symbols $<,\,>$, and further delete
the vertex associated with $i_p$-th $\times$ of $D_\lambda$ such that the $p$-th $\times$ of
$D_\mu$ is not located at vertex $0$ for all $p$. Relabeling the remaining vertices, we
obtain a dominant weight, denoted $\nu$. Then $b^\lambda_\mu=b^\nu_0$
(which is defined as in \eqref{SSSSSS}).
\begin{corollary}\label{coooo}
The formal character and dimension
of the finite dimensional irreducible $\fq(n)$-module $L(\L)$ are respectively given by
\begin{eqnarray}\label{1formula}
\ch L_\L &=& \sum_{\l\in \TC\L}(-1)^{|\L-\mu|} 2^{\frac{z(\mu)-z(\L)}{2}+\left\lfloor\frac{h(\mu)+1}{2}\right\rfloor}
b^\L_\mu P_\mu,\\ \label{1dim}
{\rm dim\,} L(\L) &=& \sum_{\mu\in \TC\L}\sum_{B\subset\Phi_{\bar0}^+}
(-1)^{|\L-\mu|} 2^{\frac{z(\mu)-z(\L)}{2}+\left\lfloor\frac{h(\mu)+1}{2}\right\rfloor}
\frac{b_\mu^\L}{\#S_\mu}  \nonumber\\
&  &  
\times (-1)^{|B\cap\Phi_{\bar0}^+(\mu)|}
\prod_{\a\in\Phi_{\bar0}^+}\frac{(\a, \mu+\sum_{\b\in B}\b+\rho_{\bar0})}{(\a, \rho_{\bar0})}.
\end{eqnarray}
\end{corollary}
\begin{proof}
The result follows from the definition of $b^\lambda_\mu$ and Theorem \ref{main-theo}.
\end{proof}



\begin{thebibliography}{9999}

\bibitem{B} J.~Brundan, Kazhdan-Lusztig polynomials and
   character formulae for the Lie superalgebra $gl(m|n)$.
   {\sl J.~Amer.~Math.~Soc.~ \bf 16} (2003) 185-231.

\bibitem{B1} J.~Brundan, {\em Kazhdan-Lusztig polynomials and character
formulae for the Lie superalgebra $\fq(n)$},  Adv.~Math. {\bf 182}  (2004) 28-77.

\bibitem{BK} J. Brundan, A. Kleshchev, Modular representations of the supergroup $Q(n)$.
 {\sl J. Algebra \bf 260} (2003), no. 1, 64--98.

\bibitem{BS} J. Brundan, C. Stroppel, Highest weight categories arising from Khovanov's diagram algebra
I, II, III, IV.  {\sl Moscow Math. J. \bf 11} (2011) 685-722;  {\sl  Transform. Groups \bf 15} (2010) 1-45;
{\sl Represent.
Theory \bf15} (2011) 170-243; {\sl  J. Eur. Math. Soc. \bf14} (2012) 373-419.

\bibitem{CL} Cheng, Shun-Jen; Lam, Ngau.
Irreducible characters of general linear superalgebra and super duality.
{\sl Comm. Math. Phys. \bf 298} (2010) 645--672.

\bibitem{CLW} Cheng, Shun-Jen; Lam, Ngau; Wang, Weiqiang.
Super duality and irreducible characters of orthosymplectic
Lie superalgebras. {\sl Invent. Math. \bf 183}  (2011)  189--224.

\bibitem{CWZ} Cheng, Shun-Jen; Wang, Weiqiang; Zhang, R. B.
Super duality and Kazhdan-Lusztig polynomials.
{\sl Trans. Amer. Math. Soc. \bf 360} (2008), no. 11, 5883--5924.

\bibitem{CZ1} S.-J. Cheng, R.B. Zhang, An analogue of
Kostant's ${\mathfrak u}$-cohomology
formula for the general linear superalgebra.
{\sl International Math Research Notices \bf1} (2004) 31-53.

\bibitem{CZ2} S.-J. Cheng, R.B. Zhang,
 Howe duality and combinatorial character formula for orthosymplectic Lie superalgebras.
{\sl Adv. Math. \bf 182}  (2004), no. 1, 124 --172.

\bibitem{GS} C. Gruson, V. Serganova, Cohomology of generalized supergrassmannians and character
formulae for basic classical Lie superalgebras. {\sl  Proc. London Math. Soc. \bf101} (2010) 852-892.

\bibitem{HKV} J.W.B.~Hughes, R.C.~King, J.~van der Jeugt,
On the composition factors of Kac modules for the Lie superalgebras $sl(m|n)$. {\sl
J.~Math.~Phys.~\bf33} (1992) 470-491.

\bibitem{Kac} V.G.~Kac, Lie superalgebras. {\sl  Adv.~Math.~\bf 26} (1977) 8-96.

\bibitem{Kac2} V.G.~Kac, {\em Representations of classical Lie superalgebras},
  Lect. Notes Math. {\bf676} (1978) 597-626.

\bibitem{PS} I. Penkov, V. Serganova, Cohomology of $G/P$ for classical complex Lie supergroups $G$ and
characters of some atypical $G$-supermodules. {\sl  Ann. Inst. Fourier \bf39} (1989) 845-873.

\bibitem{PS1} I.~Penkov, V.~Serganova, Characters of
irreducible $G$-modules and cohomology of $G/P$ for the Lie supergroup $G=Q(N)$. {\sl
 J. Math. Sci. \bf 84} (5) (1997) 1382--1412.

\bibitem{PS97} I. Penkov, V. Serganova,
{\em Characters of finite dimensional irreducible $\fq(n)$-supermodules}, Lett.
Math. Phys. {\bf40} (2) (1997) 147-158.

\bibitem{Se96} V.~Serganova, Kazhdan-Lusztig polynomials and
       character formula for the Lie superalgebra $gl(m|n)$.
       {\sl  Selecta Math.~\bf 2} (1996) 607-654.

\bibitem{Se98} V.~Serganova,
  {\em Characters of irreducible representations of simple Lie superalgebras},
  Proceedings of the International Congress of Mathematicians 1998, Berlin,
  Vol.~II, Documenta Mathematica, Journal der Deutschen
  Mathematiker-Vereinigung, pp.~583-593.

\bibitem{Sch} Scheunert, M.  {\em The theory of Lie superalgebras. An introduction}.
Lecture Notes in Mathematics, {\bf 716}. S
pringer, Berlin, 1979.

\bibitem{Su2006} Y. Su, Composition factors of Kac modules for the general linear Lie superalgebras.
{\sl Math. Z.  \bf252} (2006) 731--754.

\bibitem{SZ2007} Y. Su, R.B. Zhang, Character and dimension formulae for general linear superalgebra.
{\sl  Adv. Math. \bf211} (2007) 1--33.

\bibitem{SZ2012-1} Y. Su, R.B. Zhang, Generalised Verma modules for the
orthosymplectic Lie superalgebra $osp(k|2)$.
{\sl  J. Algebra \bf357} (2012) 94--115.

\bibitem{SZ2012-2} Y. Su, R.B. Zhang,
Generalised Jantzen filtration of Lie superalgebras I.
{\sl  J. Eur. Math. Soc. \bf14} (2012) 1103--1133.

\bibitem{SZ2013}Y. Su, R.B. Zhang, Generalised Jantzen filtration of Lie superalgebras II:
the exceptional cases. arXiv:1303.4797.

\bibitem{VHKT} J.~van der Jeugt, J.W.B.~Hughes, R.C.~King,
J.~Thierry-Mieg, Character formulas for irreducible modules of
the Lie superalgebras $sl(m|n)$. {\sl J.~Math.~Phys.~\bf31} (1990) 2278-2304.

\bibitem{VZ} J.~van der Jeugt, R.B.~Zhang,  Characters and composition
 factor multiplicities for the Lie superalgebra $gl(m|n)$.
 {\sl  Lett.~Math.~Phys.~\bf 47} (1999) 49-61.

\end{thebibliography}
\end{document}